\documentclass[10pt,reqno,a4paper]{article} 

\usepackage{anysize}
\marginsize{3.0cm}{3.0cm}{2.2cm}{2.2cm}

 \usepackage{verbatim}
\usepackage[english]{babel}
\usepackage[centertags]{amsmath}
\usepackage{amsfonts,amssymb,amsthm}
\usepackage[svgnames]{xcolor}
\usepackage{comment}
\usepackage{hyperref} 
\usepackage{mathrsfs}
\usepackage[sans]{dsfont}
\usepackage{accents}

\let\OLDthebibliography\thebibliography
\renewcommand\thebibliography[1]{
  \OLDthebibliography{#1}
  \setlength{\parskip}{0pt}
  \setlength{\itemsep}{0pt plus 0.3ex} }

\numberwithin{equation}{section}

\theoremstyle{plain}
\newtheorem{theorem}{Theorem}[section]

\newtheorem{lemma}[theorem]{Lemma}

\theoremstyle{definition}

\newcommand{\N}{{\mathbb N}}
\newcommand{\R}{{\mathbb R}}

\newcommand{\T}{{\mathbb T}}
\renewcommand{\S}{{\mathbb S}}

\newcommand{\mA}{\mathcal{A}}
\newcommand{\mB}{\mathcal{B}}
\newcommand{\mC}{\mathcal{C}}

\newcommand{\mF}{\mathcal{F}}
\newcommand{\mG}{\mathcal{G}}

\renewcommand{\a}{\alpha}
\renewcommand{\b}{\beta}
\newcommand{\g}{\gamma}

\newcommand{\e}{\varepsilon}
\newcommand{\ph}{\varphi}
\newcommand{\lm}{\lambda}

\newcommand{\Om}{\Omega}

\newcommand{\s}{\sigma}

\renewcommand{\th}{\vartheta}

\newcommand{\la}{\langle}
\newcommand{\ra}{\rangle}
\newcommand{\pa}{\partial}
\renewcommand{\div}{\mathrm{div}\,}
\newcommand{\grad}{\nabla}

\newcommand{\curl}{\operatorname{curl}}

\newcommand{\bcb}{\begin{color}{blue}}
\newcommand{\bcr}{\begin{color}{red}}
\newcommand{\bcg}{\begin{color}{green}}
\newcommand{\ec}{\end{color}}

\newcommand{\mgt}{\ph} % m.g.t. = Modulo del Gradiente della Torsione
\title{Rigidity for capillary liquid drops of nearly circular section with constant vorticity}
\author{\normalsize{Pietro Baldi, Domenico Angelo La Manna, Giuseppe La Scala}}
\date{} 

\begin{document}
\maketitle 

\abstract{We consider time-independent solutions with constant vorticity
of the free boundary Euler equations for a 3D liquid drop with capillarity. 
A rigidity result for the solutions of this problem has been recently proved 
%by the authors 
with variational methods: 
if a certain quantity involving the vorticity parameter, the capillarity coefficient 
and the area of the equatorial section of the drop is below a certain value, 
then the solution has necessarily cylindrical symmetry, 
the shape of the drop is an oblate spheroid, 
flattened at the poles and bulged at the equator, 
and each fluid particle moves along a horizontal, circular trajectory 
with constant angular velocity. 
In this paper we develop a perturbation analysis of the problem 
for fluid domains whose equatorial section is close in $C^2$ norm 
to a disc, and we show that a rigidity result holds also above 
the threshold obtained with variational methods.}
 
%Via a perturbative argument we show that 
%if $\Omega$ is a stationary rotating solution, 
%the ratio $\frac{\a_0^2}{\s_0}$ is less than $\sqrt2(\frac{\pi}{|D|})^\frac32$
%and the equatorial section is a nearly disk, then
 %$\Omega$ has cylindrical symmetry.

\section{Introduction and main result}

We consider the free-boundary problem for a capillary liquid drop with constant vorticity
\begin{equation}\label{eq:system}
\begin{cases}
%\mD_t u = -\nabla p 
\pa_t u + \la u , \grad \ra u + \grad p = 0 
& \text{in } \Omega_t,
\\
\div u=0 & \text{in }  \Omega_t,
\\ 
\curl u=\a_0 e_3 & \text{in } \Om_t,
\\
p = \sigma_0 H_{\pa \Omega_t} & \text{on } \partial \Omega_t,
\\
V_t=\la u,\nu_{\pa\Om_t}\ra & \text{on } \partial \Omega_t,
\end{cases}
\end{equation}
where $\Om_t \subset \R^3$ is an open bounded set, occupied by the fluid,  
$u$ is the fluid velocity vector field, 
$p$ is the fluid pressure, 
$\alpha_0$ is a vorticity parameter, 
$e_3 = (0,0,1)$,
$\sigma_0$ is the capillarity constant, 
$H_{\pa \Omega_t}$ is the mean curvature of the boundary $\pa \Omega_t$ , 
$V_t$ is the normal velocity of the boundary $\pa \Om_t$, 
and $\nu_{\pa \Om_t}$ is the outer unit normal to the boundary.
The unknowns of the problem are $\Om_t, u, p$. 
In fact, as is well-known, the pressure $p$ is completely determined by $\Om$ and $u$, 
and for this reason we can consider $(\Om, u)$ as the only unknowns. 

We are interested in time-independent solutions of \eqref{eq:system}, 
for which the system %\eqref{eq:system} 
becomes  
\begin{equation}\label{eq:system.time.indep}
\begin{cases}
\la u , \grad \ra u + \grad p = 0 
& \text{in } \Omega,
\\
\div u=0 
& \text{in }  \Omega,
\\ 
\curl u=\a_0 e_3 
& \text{in } \Om,
\\
p = \sigma_0 H_{\pa \Omega} 
& \text{on } \partial \Omega,
\\
\la u, \nu_{\pa\Om} \ra = 0 
& \text{on } \partial \Omega.
\end{cases}
\end{equation}
The study of this problem dates back to the experimental work of Plateau \cite{Plateau}, 
where uniformly rotating surfaces of revolutions were observed, 
and the theoretical ones of Poincaré \cite{Poincaré} and Rayleigh \cite{Rayleigh}, 
where their existence is formally shown. 
What arises from these works is that the behaviour of such solutions 
is intimately connected with the quotient $\a_0^2 / \s_0$, 
as is also shown in \cite{Lopez} through a variational characterization of them. 
In \cite{BLL.1} it is shown that, under the convexity assumption for the domain $\Om$, 
any solution of \eqref{eq:system.time.indep} has horizontal velocity vector field, 
i.e., $\la u,e_3\ra = 0$, both $u$ and $p$ are independent of the variable $x_3$,  
the set $\Om$ has a horizontal plane of symmetry, which, 
with no loss of generality, 
%assuming the barycenter of the fluid mass is at the origin, 
is the plane $x_3 = 0$, 
and, calling $D \subset \R^2$ the equatorial section, i.e.,
\begin{equation} \label{def.D}
\Om \cap \{ x_3 = 0 \} = D \times \{ 0 \},   
\end{equation}
the boundary $\pa \Om$ in the half-space $x_3 > 0$ 
is the graph of a positive function $f : D \to (0, \infty)$, 
%i.e., $\pa \Om \cap \{ x_3 > 0 \} = \{ (x', f(x')) : x' \in D \}$. 
\begin{equation} \label{graph.f}
\pa \Om \cap \{ x_3 > 0 \} 
%= \{ (x', f(x')) : x' \in D \}.
= \{ (x_1, x_2, f(x_1, x_2)) : (x_1, x_2) \in D \}.
\end{equation}
It is also proved that the profile function $f$ is constant on the level sets of the
\emph{torsion function} $v_D$ of the planar set $D$, which is, by definition, 
the solution of the problem 
\begin{equation} \label{torsion.pb}
\Delta v = -1 \quad \text{in } D, \qquad 
v = 0 \quad \text{on } D.
\end{equation}
It is proved that the torsion function $v_D$ also has to satisfy the identity 
\begin{equation} \label{eq:torsion.3} 
\frac{\alpha_0^2}{2}  |\nabla v|^{2} 
+ \frac{ \sigma_0 |\beta|}{ |\nabla v| }  
+ \sigma_0 H_{\pa D} = c 
\quad \text{on } \pa D,
\end{equation}
where $\beta, c$ are real constants, 
and $H_{\pa D}$ is the curvature of the planar curve $\pa D$. 
Thus \eqref{torsion.pb}, \eqref{eq:torsion.3} is an overdetermined problem for the set $D$. 
The constants $\b, c$ can be expressed in terms of the planar set $D$, 
and one has $- \beta = |\beta| \geq 1/2$, 
\begin{equation} \label{eq:syst.beta.c,1}
c = \frac{ 3\alpha^2_0 Q + 2 \sigma_0 P }{2|D|},
\quad \ 
|\beta| = \frac{ \sigma_0 P + \alpha_0^2 Q }{\sigma_0 I_R}
= \frac{ \alpha_0^2 ( 3 Q - I_C ) + 2 \sigma_0 ( P - I_H ) }{2 \sigma_0 P},
\end{equation}
where $|D|$ is the area of $D$, 
$P = P(D)$ is its perimeter, and 
$Q, I_R, I_C, I_H$ are the shape functionals 
\begin{alignat}{2}
Q(D) & := \int_{D} |\nabla v_D|^2 \, dx, 
\qquad \ &
I_R(D) & := \int_{\pa D} \frac{ \la x,\nu_{\pa D} \ra }{ |\nabla v_D| } \, d\sigma, 
\label{def.Q.IR} 
\\ 
I_C(D) & := \int_{\pa D} |\nabla v_D|^3 \, d\sigma,
\qquad \ & 
I_H(D) & := \int_{\pa D} H_{\pa D} |\nabla v_D| \, d\sigma.
\label{def.IC.IH}
\end{alignat} 
In \cite{BLL.1} it is proved that, if 
\begin{equation}\label{convex.bound}
\frac{\a_0^2}{\s_0} |D|^{\frac32} < \sqrt{2} \, \pi^{\frac32},
\end{equation}
then $D$ can only be a disc, 
and hence the solution of \eqref{eq:system.time.indep} has cylindrical symmetry, 
$\Omega$ is an oblate spheroid, flattened at the poles and bulged at the equator, 
whose profile function $f$ in \eqref{graph.f} 
is the radial function $f(x_1, x_2) = f_0(\xi)$, 
$\xi = \sqrt{x_1^2 + x_2^2} \in [0, r]$, 
with $f_0(r) = 0$, 
\begin{equation} \label{formula.f}
\frac{f_0'(\xi)}{\sqrt{1 + f_0'(\xi)^2}} 
= - \frac{ (1 - \gamma) \xi}{r} - \frac{\gamma \xi^3}{r^3}, 
\quad \ \gamma = \frac{\alpha_0^2 r^3}{32 \sigma_0}, 
\quad \ r = \Big( \frac{|D|}{\pi} \Big)^{\frac12},
\end{equation}
and each fluid particle moves along a horizontal, circular trajectory 
with constant angular velocity. 
This rigidity property comes from the fact that when \eqref{convex.bound} holds, 
the problem \eqref{torsion.pb}, \eqref{eq:torsion.3} can be turned into 
an overdetermined Serrin problem \cite{S}, 
in which \eqref{eq:torsion.3} is a constraint of the form  
$|\grad v_D| = \ph (H_{\pa D})$, 
where $\ph$ is a function with the monotonicity property 
required by Serrin's Theorem. 
Hypothesis \eqref{convex.bound} is assumed in \cite{BLL.1} 
precisely to guarantee that monotonicity.

The purpose of the present paper is to investigate further 
the range of parameters where this rigidity property occurs, 
using a perturbation analysis instead of variational methods.
We express the boundary $\pa D$ as a graph over the unit circle 
$\S^1 := \{ x \in \R^2 : |x|=1 \}$ of an elevation function, %  $h$, 
we calculate the second order Taylor expansion of the shape functionals 
around the circular shape, %  when $D$ is a disc, 
and we prove that, when $D$ is close to a disc in $C^2$ topology, 
rigidity occurs even if condition \eqref{convex.bound} is violated.   
This shows that bound \eqref{convex.bound} 
%although rather natural to impose  
is not optimal. 

To state the main result, it is convenient to introduce the following notation. 
For any real $m > 0$, any planar set $D$ of area $|D|=m$ 
%with barycenter at the origin of $\R^2$ 
can be written as $D = r D'$ where $D'$ has the same area as the unit disc 
$|D'| = \pi$. Thus $m = |D| = |rD'| = r^2 \pi$, that is, 
%i.e., $r = (\frac{m}{\pi})^{\frac12}$. 
\begin{equation} \label{D.D'}
D = r D', \quad \ 
|D'| = \pi, \quad \ 
r = \Big( \frac{|D|}{\pi} \Big)^{\frac12}.
\end{equation}
Assuming that the boundary of $\pa D'$ is the graph of an elevation function $h$ over 
the unit circle $\S^1$, one has 
\begin{equation} \label{pa.D.graph} % {nearly.circular}
\pa D' = \{ x (1 + h(x)) : x \in \S^1 \}, \quad \ 
\pa D =  \{ r x (1 + h(x)) : x \in \S^1 \},
\end{equation}
with $r$ in \eqref{D.D'}. % where $r = (|D|/\pi)^{1/2}$.
The main result of this paper is the following theorem. 

\begin{theorem}[Rigidity of solutions with nearly circular section] 
\label{thm:rigidity} 
Given $\alpha_0, \sigma_0 > 0$, 
there exists $\delta > 0$ 
depending on $\alpha_0, \sigma_0$ with the following property. 
Let $(u,\Om)$ be a $C^2$ solution of \eqref{eq:system.time.indep} 
such that the barycenter of the equatorial section $D$ in \eqref{def.D} 
is at the origin of $\R^2$ 
and its boundary $\pa D$ is the graph \eqref{pa.D.graph} 
of an elevation function $h$ with 
\[
\| h \|_{C^2(\S^1)} < \delta.
\] 
Suppose that 
\begin{equation} \label{Rayleigh.ratio}
\frac{\a_0^2}{\sigma_0} \left(\frac{|D|}{\pi}\right)^\frac32 < \frac{64}{3}.
\end{equation}
Then $D$ is a disc, $\Omega$ is the oblate spheroid of profile \eqref{graph.f}, \eqref{formula.f},
and $u$ is the vector field $\frac12 \alpha_0 (-x_2, x_1, 0)$. 
%and each fluid particle moves along a horizontal, circular trajectory with constant angular velocity. 
\end{theorem}

\medskip

A planar set $D$ like that in Theorem \ref{thm:rigidity} 
is close to a disc in the $C^2$ topology. 
We call it a \emph{nearly circular} set.

\medskip

\underline{Summary of the proof}. 
We present, in short, the proof of Theorem \ref{thm:rigidity}.  
%The main argument leading to Theorem \ref{thm:rigidity} is the following. 

\medskip

\emph{The key identity.}
From the last identity in \eqref{eq:syst.beta.c,1}, 
separating the terms with the vorticity parameter $\alpha_0^2$ 
from those with the capillarity coefficient $\sigma_0$,
we find 
\begin{equation} \label{main.identity}
\s_0 \mA(D) + \frac{\a_0^2}{2} \mB(D) = 0,
\end{equation}
where 
\begin{equation} \label{eq:defn.mA}
\mA := (P - I_H) I_R - P^2, \qquad 
\mB := I_R (3Q-I_C) - 2PQ.
\end{equation}
Identity \eqref{main.identity} is the basis for the proof of Theorem \ref{thm:rigidity}.

\medskip

\emph{Normalization of the area.}
The functionals $\mA,\,\mB$ satisfy the following invariance properties.
Translation: if $v_E$ is the torsion function of some planar set $E$ and
$E'= x_0 + E$, then $v_{E'}(x) = v_{E}(x-x_0)$.
Hence both $\mA$ and $\mB$ are invariant under translations. 
%and assuming that $D$ satisfies the barycenter of the set $D$ is at the origin. 
Scaling: for $r > 0$, the torsion function of the set $r E$ 
is $v_{r E} (x) = r^2 v_E(\frac xr)$,
and therefore $\mA$ and $\mB$ satisfy 
\[
\mA (r E) = r^2 \mA(E), \quad \ 
\mB (r E) = r^5 \mB(E).
\]
These invariance properties enables one to split the role of the area $|D|$ from that of the shape of $D$. 
Writing the set $D$ as in \eqref{D.D'}, 
%For any given $m > 0$, any set $D$ of area $|D|=m$ 
%can be written as $D = r D'$ where $D'$ has the same area as the unit disc 
%$|D'| = \pi$. Thus $m = |D| = |rD'| = r^2 \pi$, i.e., 
%$r = (\frac{m}{\pi})^{\frac12}$, and 
identity \eqref{main.identity} becomes 
\begin{equation} \label{mC.D'.=.0}
\mC(D') = 0,
\end{equation}
where $\mC$ is the functional 
\begin{equation} \label{def.mC}
\mC := \mA + \lm_0 \mB, 
\end{equation}
and the constant $\lm_0$ is 
\[
\lm_0 := \frac{\alpha_0^2}{2 \sigma_0} \Big( \frac{|D|}{\pi} \Big)^{\frac32}.
\]
Note that \eqref{mC.D'.=.0} is an identity for the set $D'$ 
with normalized area $|D'| = \pi$, 
while the area $|D|$ appears in the parameter $\lm_0$. 
This normalization is not merely convenient, but it plays a role 
in the last part of the proof. 

%In the proof of Theorem \ref{thm:rigidity}, we study $\mC(D')$ for sets $D'$ of area $|D'|=\pi$, 
%for any $\lm_0 > 0$. We drop the prime, and write $|D| = \pi$. 
%following, for any given $m$, we study the functional 
%$\mC(D')$ among smooth sets $D'$ of area $|D'|=\pi$. 
%We drop the prime, and write $|D| = \pi$. 
%where we have set 
%\[
%k_0=\frac{2 \sigma_0}{\alpha_0 ^2 r_m^3}.
%\]
%Thus since \eqref{main.identity} is equivalent to find the sets $D$ such that $|D|=\pi$ and $\mC(D)=0$, 
%the assumption that $D=D_\e$ is a nearly-disk for some $\e>0$ 
%allows to expand the functional $\mC$ in terms of the perturbation parameter $\e$ 
%and the graph function $h$ as in the Definition \ref{defn:nearly.spherical}. 
%Using perturbative analysis, we show that
%\[
%\mC(D_\e)\geq \frac{\pi}{k_0}(32 k_0^2-3) |D_\e \triangle D|^2,
%\]
%where $|D_\e \triangle D|= |D\setminus D_\e|+|D_\e \setminus D|$. 
%Therefore, if \eqref{Rayleigh.ratio} is satisfied, then $\mC(D_\e)\geq 0$
%with equality holding if and only if $D_\e$ is a disk,
%and thus the problem \eqref{eq:torsion} does not admit a solution, unless $D_e$ is a disk.

\medskip

\emph{Regularity, analyticity and Taylor remainders.}
We study the normalized identity \eqref{mC.D'.=.0} by means of its Taylor expansion 
around the case when $D'$ is the unit disc. 
We drop the prime, and write $\mC(D)$, $|D| = \pi$.  

First, iterating a bootstrap argument from \cite{BLL.1} 
based on identity \eqref{eq:torsion.3} and Schauder estimates, 
we observe that the boundary $\pa D$ is smooth, 
and the high norms of the elevation function $h$ are controlled 
by its $C^2$ norm (Lemma \ref{lemma:smooth}). 
Then we prove a formula, of independent interest, for the modulus of the gradient 
of the torsion function at the boundary in terms of the Dirichlet-Neumann map 
(Lemma \ref{lemma:torsion.DN}), and from this we quickly deduce its analytic dependence 
on $h$ in Sobolev norm (Lemma \ref{lemma:torsione.analitica}) 
and a cubic estimate for its Taylor remainder in norm $C^2$ (see \eqref{est.R.base}).  

\medskip

\emph{Taylor polynomials}. 
After estimating the Taylor remainders, 
we calculate the Taylor polynomial of order 2 
of the curvature (Lemma \ref{lemma:J.powers.expansion}),  
that of the gradient of the torsion function at the boundary (Lemma \ref{cor:v.t.taylor.exp}), 
which is obtained after calculating certain shape derivatives (Lemma \ref{lem:derivative}), 
and that of the functionals $\mA$ and $\mB$ (Lemma \ref{lemma:integral.perturbative.exp}).  
From these expansions we obtain the Taylor polynomial $T(h, \mC)$ of order 2 
of the functional $\mC$ as a function of the elevation $h$.  
Expanding $h$ in Fourier series (or, equivalently, in spherical harmonics), we find 
\[
T(h, \mC) = 2 \pi \sum_{\begin{subarray}{c} \ell \geq 2 \\ m=1,2 \end{subarray}} 
\ell (\ell -1) \Big(\ell+2-\frac{3 \lm_0}{8} \Big) h^2_{\ell,m},
\]
where $h_{\ell,m}$ are the coefficients of the expansion of $h$, 
see \eqref{Taylor.mC}. 
Thus, the first nonzero contribution to the Taylor expansion of $\mC$ around $h=0$ 
is quadratic in $h$. %  (no term linear in $h$ is present.  

\medskip

\emph{Quadratic lower estimate.} %  Conclusion of the proof.} 
For $(4 - \frac38 \lm_0) > 0$, 
which corresponds to condition \eqref{Rayleigh.ratio} in Theorem \ref{thm:rigidity}, 
the products $2\pi \ell (\ell -1) (\ell+2-\frac{3 \lm_0}{8} )$ 
are all positive, and they are $\geq \mu_0 \ell^3$ for some constant $\mu_0 > 0$, 
for all $\ell \geq 2$.
%The coefficients $h_{\ell,m}$ with $\ell=0, 1$, instead, do not appear 
%in the formula of $T(h, \mC)$. 
%This is a consequence of the invariances 
Thus, decomposing $h$ as the sum $h = a + b$, 
where $a$ contains all the modes $\ell \geq 2$ 
and $b$ those with $\ell = 0, 1$, see \eqref{def.a.b.h}, 
the Taylor expansion and the estimates proved so far give 
\[
\mC(D) \geq \mu_0 \| a \|_{H^{\frac32}(\S^1)}^2 - C \| h \|_{C^2(\S^1)}^3,
\]
see \eqref{end.arg.06}. 
Now a basic principle says that, for $h$ small, 
the quadratic term should dominate the cubic one.
%and therefore the identity $\mC(D) = 0$ should imply that $h = 0$. 
However, this principle does not directly apply for two reasons: 
\begin{itemize}
\item[$(i)$] $b$ is missing in the quadratic term, but it is present in the cubic one;
\item[$(ii)$] the $C^2$ norm is stronger than the $H^{\frac32}$ norm.
\end{itemize}
Hence we proceed in this way. 
For the $a$ component, 
we use the interpolation property of Sobolev norms, 
combined with the control of high norms of $h$ 
in terms of its $C^2$ norm in Lemma \ref{lemma:smooth},  
to get 
\[ 
\| a \|_{C^2(\S^1)} 
\leq C \| a \|_{H^{\frac32}(\S^1)}^{\frac23} 
\| h \|_{C^2(\S^1)}^{\frac13},
\]
see \eqref{end.arg.01}. 
For the % Concerning item $(i)$ and the 
$b$ component, we prove that 
\[
\| b \|_{C^2(\S^1)} 
\leq C \| b \|_{L^2(\S^1)} 
\leq C' \| a \|_{L^2(\S^1)}^2,
\]
see \eqref{end.arg.02}, \eqref{end.arg.03}. 
This estimate, which is linear in $b$ and quadratic in $a$, 
is directly related to the geometric assumptions on $D$: 
for the mode $\ell=0$, this is due to the area condition $|D| = \pi$, 
see \eqref{h.integral} and \eqref{storz.01}; 
for the mode $\ell=1$, it is due to the assumption 
that $D$ and the unit disc have the same barycenter, 
see \eqref{eq:ell=1}. 
From these estimates we obtain the quadratic lower estimate 
\[
\mC(D) \geq \frac{\mu_0}{2} \| a \|_{H^{\frac32}(\S^1)}^2
\]
for $\| h \|_{C^2(\S^1)}$ sufficiently small.  
Thus $\mC(D) = 0$ implies that $a=0$, whence $h=0$, 
namely $D$ is a disc. 
Since $D$ is a disc, $\Om$ and $u$ are easily determined (Remark 4.4 in \cite{BLL.1}). 

\medskip

\underline{Related literature}.
Axisymmetric equilibria for 3D capillary drops with constant vorticity 
were discovered experimentally by Plateau in \cite{Plateau} 
and studied by 
Poincaré \cite{Poincaré}, 
Rayleigh \cite{Rayleigh}, 
Chandrasekhar \cite{Chandrasekhar}, 
and, more recently, 
Wente \cite{Wente} 
and Lopez \cite{Lopez}. 

Free-boundary problems for fluids with constant vorticity have been studied 
in \cite{Wahlen.2007, Wahlen.2014, Pasquali} for the ocean, 
and \cite{LaS, BLL.1} for the drop.
Global solutions for fluids with constant vorticity have been obtained in 
\cite{Constantin.Strauss, Constantin.K.2009, Wahlen.2006, Fan.Gao.2021}, 
and in \cite{BBMM, BFM} for periodic and quasi-periodic travelling waves. 
Travelling drops and bubbles in a two-fluid interaction have been constructed in \cite{MNS}, 
where an overdetermined problem also appears. 
Travelling rotating waves for the capillary drop are studied in \cite{BJL, BLL, LaS.bif, LaS}.

Rigidity results for the 3D fluid dynamics 
are in \cite{Wahlen.2014} for the ocean with constant vorticity, 
in \cite{Peralta.Salas} for localizable Euler flows in bounded domains,  
and in \cite{BLL.1} for capillary drops with constant vorticity.

\medskip

\emph{Acknowledgements.} This work is supported by GNAMPA 
and by University of Naples Federico II through FRA 2024 \emph{Geometric Topics in Fluid Dynamics}.

\section{Proof}

To shorten the notation, we write $v, \nu$ instead of $v_{D}, \nu_{\pa D}$.  
The unit normal $\nu$ to the curve $\pa D$ at the point 
$\gamma(x) = x(1+h(x)) \in \pa D$, with $x \in \S^1$, is 
\begin{equation}\label{nu.epsilon.parametrization}
\nu \circ \gamma = \frac{(1+h)x - \grad_{\S^1} h}{J},
\quad \ 
J = J(h) :=	\sqrt{(1+h)^2 + |\grad_{\S^1} h|^2}
\end{equation}
(see, e.g., \cite{BJL}). 
%where 
%\begin{equation}
%\label{J.def}
%J = J(h) :=	\sqrt{(1+h)^2 + |\grad_{\S^1} h|^2}.
%\end{equation}
We denote $G(h) \psi$ the Dirichlet-Neumann operator 
\[
G(h)\psi := \la (\grad \Psi) \circ \gamma , \nu \circ \gamma \ra,
\]
where $\Psi$ is the unique solution of the Laplace problem with Dirichlet datum
\[
\Delta \Psi = 0 \ \text{in } D, \qquad 
\Psi = \psi \circ \gamma^{-1} \ \text{on } \pa D.
\]
For $h=0$, we denote $\mG := G(0)$. % the Dirichlet-Neumann operator of the unit disc $D_0$. 
The Dirichlet-Neumann operator for nearly spherical domains has been studied 
in \cite{BJL, BJL.Dir.Neu, LaS} in Sobolev class, 
including its analytical dependence on the elevation function $h$; 
estimates in H\"older spaces, including analyticity 
(convergence of series of multi-linear operators) are similar 
and, in some respects, slightly easier to prove. 
In particular, $\mG$ is a bounded linear operator of $C^{k+1,\alpha}(\S^1)$ into $C^{k,\alpha}(\S^1)$ 
for all nonnegative integer $k$, all $\alpha \in (0,1)$. 
%(unlike Sobolev spaces, $C^{k, \alpha}(\S^1)$ are closed for the product of two functions even for $k=0$). 
%the Dirichlet-Neumann operator of the unit disc $D_0$, namely
%\[
%(\mG \psi)(x) := \la \grad \Psi(x) , x \ra, \quad \ x \in \S^1,
%% \la \grad \Psi \circ \g_0,\nu_0\circ\g_0\ra
%\]
%where $\Psi$ is the unique bounded solution of the Dirichlet problem 
%\begin{equation*}\begin{cases}
%\Delta\Psi=0 
%& \text{in } D_0,
%\\
%\Psi=\psi 
%& \text{on } \S^1.
%\end{cases}
%\end{equation*}

Before starting with calculating Taylor expansions, 
we recall a higher regularity result from \cite{BLL.1}, 
obtained from \eqref{eq:torsion.3} by bootstrap, 
and we prove an identity involving the torsion function and the Dirichlet-Neumann operator.

\begin{lemma}[Solutions are smooth] 
\label{lemma:smooth}
Let $\Om \subset \R^3$ be a convex, open, bounded set with $C^2$ boundary $\pa \Om$. 
Let $u$ be a vector field in $\Om$ of class $C^2(\Om) \cap C(\overline{\Om})$. 
Assume that $(\Om, u)$ is a solution of \eqref{eq:system.time.indep}. 
Then the planar set $D$ has $C^\infty$ boundary, 
and its torsion function $v_D$ is $C^\infty(\overline{D})$. 
In particular, there exist $C,\delta > 0$ such that 
\[
\| h \|_{C^6(\S^1)} \leq C \| h \|_{C^2(\S^1)}
\]
for all functions $h$ in the ball $\| h \|_{C^2(\S^1)} \leq \delta$.
\end{lemma}

\begin{proof} 
In \cite{BLL.1} it is proved that, when $(\Om, u)$ solves \eqref{eq:system.time.indep}, 
the torsion function of the equatorial section $D$ 
and the curvature of its boundary satisfy identity \eqref{eq:torsion.3}, 
which gives $H_{\pa D}$ as a function of $|\grad v_D|$. 
Moreover $C_0 \leq |\grad v_D| \leq C_1$ on $\pa D$, 
where $C_0, C_1$ are positive constants (Lemma 4.7 in \cite{BLL.1}). 
Hence the bootstrap argument at the end of the paper \cite{BLL.1}, 
based on Schauder regularity theory, 
can be iterated $n$ times, for any $n \in \N$. 
The estimate of the norm $\| h \|_{C^{k, \alpha}(\S^1)}$ for any $k$ 
in terms of $\| h \|_{C^2(\S^1)}$ can be proved by using the same ingredients, 
namely Schauder estimates and \eqref{eq:torsion.3}. 
\end{proof}

The $C^6$ norm of $h$ will be used 
in an interpolation argument at the end of the paper. 
In the next lemma we observe that the modulus of the gradient of the torsion function at the boundary 
can be easily expressed in terms of the elevation function $h$ and its Dirichlet-Neumann operator. 

\begin{lemma}[A formula for the torsion function and the Dirichlet-Neumann operator]
\label{lemma:torsion.DN}
The gradient of the torsion function $v$ at the boundary $\pa D$ 
and the Dirichlet-Neumann operator 
are related by the identity 
\[
|(\grad v) \circ \gamma| = \frac{(1+h)^2}{2J} - G(h)\psi, \quad \ 
\psi := \frac{(1+h)^2}{4} \quad \ \text{on } \S^1.
\]
\end{lemma}

\begin{proof}
Let $v := v_D$. 
Let $w(x) := |x|^2/4$, $x \in \R^2$, 
and let $u := v+w$ in $D$. 
Then $\Delta u = \Delta v + \Delta w = 0$ in $D$, 
and, on the boundary,  
$u(x) = w(x) + v(x) = |x|^2 / 4$ for $x \in \pa D$. 
Hence 
\[
u \circ \gamma = \frac{ |\gamma|^2 }{4} = \frac{ (1 + h)^2 }{4} = \psi 
\quad \text{on } \S^1,
\]
and, by definition of the Dirichlet-Neumann operator, 
\[
\la (\grad u) \circ \gamma , \nu \circ \gamma \ra = G(h)\psi 
\quad \text{on } \S^1.
\]
On the other hand, $\grad u = \grad v + \grad w$.  
The torsion function $v$ satisfies $\nu = - \grad v / |\grad v|$ on $\pa D$, whence 
\[
\la (\grad v)\circ \gamma , \nu \circ \gamma \ra = - | (\grad v) \circ \gamma| 
\quad \text{on } \S^1.
\]
The function $w$ satisfies $\grad w(x) = x/2$ everywhere in $\R^2$, whence 
\[
\la (\grad w)\circ \gamma , \nu \circ \gamma \ra  
= \la \frac{\gamma}{2} , \nu \circ \gamma \ra
= \la \frac{ (1+h) x }{2} , \frac{(1+h) x - \grad_{\S^1} h}{J} \ra 
= \frac{(1+h)^2}{2J}
\quad \text{on } \S^1.
\]
From the identities above, the lemma follows. 
\end{proof}

The gradient of the torsion function at the boundary has a central role 
in the integral quantities we have to analyze. 
To indicate its dependence on the elevation function $h$, we define 
\begin{equation} \label{def.mgt}
\mgt(h) := |(\grad v) \circ \gamma| \quad \text{on } \S^1,
\end{equation}
where $v$ is the torsion function of the set $D$, 
and $\pa D = \gamma(\S^1)$, $\gamma(x) = x (1 + h(x))$.   
From the analyticity of the Dirichlet-Neumann operator in \cite{BJL.Dir.Neu}
and Lemma \ref{lemma:torsion.DN} we immediately deduce 
the analyticity of $\mgt$ in Sobolev class. 
A similar result could also be proved in H\"older spaces, 
but here, in fact, we do not need it, 
because, thanks to Lemma \ref{lemma:smooth}, 
it is more convenient to recover the H\"older regularity from the Sobolev one.

\begin{lemma}[Analyticity of $\ph$]
\label{lemma:torsione.analitica}
There exist $\delta_0 > 0$ such that, for any real $s \geq 1$, for any $h$ in the set 
\[
U := \{ h \in H^{s+1}(\S^1) : \| h \|_{H^2(\S^1)} < \delta_0 \},
\]
the function $\mgt(h)$ is in $H^s(\S^1)$. 
The constant $\delta_0$ is independent of $s$. 
The map $\mgt : h \mapsto \mgt(h)$ is analytic 
from $U$ (with the norm $\| \ \|_{H^{s+1}(\S^1)}$) to $H^s(\S^1)$. 
\end{lemma}

\begin{proof} 
Apply Theorem 1.1 of \cite{BJL.Dir.Neu} with $n=2$, $s_0 = 1$, $s \geq 1$ 
for the analyticity of the Dirichlet-Neumann operator,  
Lemmas 2.6 and 4.11 of \cite{BJL.Dir.Neu} to estimate $h^2$ and $J^{-1}$, 
and Lemma \ref{lemma:torsion.DN}.% for the formula of $\mgt(h)$. 
\end{proof}

From Lemma \ref{lemma:torsione.analitica}, 
by the standard theory of analytic maps between Banach spaces,  
we immediately obtain the Taylor expansion in Sobolev spaces 
\begin{equation} \label{Taylor.mgt}
\mgt(h) = \mgt(0) + \mgt'(0)[h] + \frac12 \mgt''(0)[h,h] + R(h), 
\quad \ 
\| R(h) \|_{H^s(\S^1)} \leq C \| h \|_{H^{s+1}(\S^1)}^3 
\end{equation} 
for all $\| h \|_{H^{s+1}(\S^1)} \leq \delta_s$, for some $\delta_s < \delta_0$.  
We observe that from \cite{BJL.Dir.Neu} a stronger, ``tame'' estimate 
for the remainder $R(h)$ could be obtained, 
but here, also thanks to Lemma \ref{lemma:smooth}, 
\eqref{Taylor.mgt} is sufficient for our purpose.
In fact, by Sobolev embedding, one has 
$\| R(h) \|_{C^0(\S^1)} \leq C \| R(h) \|_{H^1(\S^1)}$, 
then we apply estimate \eqref{Taylor.mgt} with $s=1$, 
and then the trivial inequality 
$\| h \|_{H^2(\S^1)} \leq C \| h \|_{C^2(\S^1)}$, 
and we obtain that the Taylor remainder of $\mgt$ satisfies
\begin{equation} \label{est.R.base}
\| R(h) \|_{C^0(\S^1)} \leq C \| h \|_{C^2(\S^1)}^3. 
\end{equation}

Also $J : U \to H^s(\S^1)$ 
and $H_{\pa D} \circ \gamma : U \to H^{s-1}(\S^1)$ 
are analytic functions of $h$, 
as one can immediately deduce from the formula of $J$ in \eqref{nu.epsilon.parametrization} 
and from that of the curvature 
\begin{equation}\label{mean.curvature.def}
H_{\pa D} \circ \gamma
= \frac1{J} 
- \frac{\Delta_{\S^1}h}{(1+h) J} 
+ \frac{|\grad_{\S^1}h|^2}{J^3} 
+ \frac{\la(D_{\S^1}^2 h) \grad_{\S^1} h, \grad_{\S^1} h \ra}{(1+h) J^3},
\end{equation}
see, e.g., \cite{BJL}. 
Hence $J$ and $H_{\pa D} \circ \gamma$ satisfy a similar Taylor expansion as \eqref{Taylor.mgt}
(for the curvature, the remainder is estimated in $H^{s-1}$ norm). 
The Taylor remainder of $J$ satisfies \eqref{est.R.base}, just like that of $\mgt$ does.
The Taylor remainder of $H_{\pa D} \circ \gamma$ satisfies, 
using also Lemma \ref{lemma:smooth}, 
\begin{align} 
\| R(h) \|_{C^0(\S^1)} 
\leq C \| R(h) \|_{H^1(\S^1)} 
& \leq C \| h \|_{H^3(\S^1)}^3 
%\notag \\ & 
\leq C \| h \|_{C^3(\S^1)}^3 
\leq C \| h \|_{C^2(\S^1)}^3,
\label{est.R.catena.1}
\end{align}  
which is, again, \eqref{est.R.base}.

%Now we start with calculating Taylor expansions. 
%To make this paper self-contained, we give here another   
%This is classical for nearly flat domains (see, e.g., \cite{Lannes.book}). 
%For nearly spherical sets, it has been proved in \ref{BJL.Dir.Neu} in Sobolev class. 
%It is also true for H\"older spaces, and the proof is very similar (in fact, slightly easier). 
%By Lemma \ref{lemma:torsion.DN}, the function $|(\grad v) \circ \gamma|$ on $\S^1$ 
%also depends analytically on $h$.  
%For functions defined on $\S^1$, we introduce the short notation 
%\[
%\| h \|_k := \| h \|_{C^{k,\alpha}(\S^1)},
%\]
%where $\alpha \in (0,1)$ is fixed. 

Now that the Taylor remainders of $\mgt, J, H_{\pa D} \circ \gamma$ of order 2 
have been proved to satisfy \eqref{est.R.base}, we have to calculate the Taylor polynomials 
\begin{equation} \label{def.Taylor.polynomial}
T(h,f) := f(0) + f'(0)[h] + \frac12 f''(0)[h,h]
\end{equation}
for $f = \mgt, J, H_{\pa D} \circ \gamma$. 

\begin{lemma}[Taylor polynomials of $J$ and $H_{\pa} \circ \gamma$]
\label{lemma:J.powers.expansion}
For $p \in \{ 1, -1, -3\}$, the Taylor polynomials of $J^p$ 
and $H_{\pa} \circ \gamma$ are 
\begin{align} 
T(h, J^p) & = 1 + p h + \frac{p(p-1)}{2} h^2 + \frac{p}{2} |\grad_{\S^1} h|^2, %  + R_p(h), 
\label{J.expansion}
\\ 
T(h, H_{\pa D} \circ \gamma) 
& = 1 - (h + \Delta_{\S^1} h) 
+ \Big( h^2 + \frac12|\grad_{\S^1}h|^2 + 2h\Delta_{\S^1}h \Big). %  + R(h),
\label{mean.curvature.perturbative.expansion}
\end{align}
%where the remainders satisfy \eqref{est.R.base} for all $\| h \|_{C^2(\S^1)} < \delta$. 
\end{lemma}

\begin{proof} 
Straighforward calculation with formulas \eqref{nu.epsilon.parametrization} of $J$ 
and \eqref{mean.curvature.def} of $H_{\pa D} \circ \gamma$.
%One has $J^p(h) = (1 + 2h + h^2 + |\grad_{\S^1} h|^2)^{p/2}$.
%Use the power series $(1+a)^q = \sum_{n=0}^\infty \binom{q}{n} a^n$ 
%with $a = 2h + h^2 + |\grad_{\S^1} h|^2$, which converges totally in norm $\| \ \|_1$ 
%for $\| h \|_2$ sufficiently small.  
%The curvature of $\pa D$ is 
%\begin{equation}\label{mean.curvature.def}
%H_{\pa D} 
%= \frac1{J} 
%- \frac{\Delta_{\S^1}h}{(1+h) J} 
%+ \frac{|\grad_{\S^1}h|^2}{J^3} 
%+ \frac{\la(D_{\S^1}^2 h) \grad_{\S^1} h, \grad_{\S^1} h \ra}{(1+h) J^3},
%\end{equation}
%see, e.g., \cite{BJL}. 
%Thus, by \eqref{J.expansion}, a simple computation yields 
%\eqref{mean.curvature.perturbative.expansion}.
\end{proof}

Now we calculate the Taylor polynomial of $\ph = |(\grad v) \circ \gamma|$. 
%  around $h=0$ of order 2. 
%Since $G(h)\psi$ is linear in $\psi$ and depends analytically on $h$, 
%by Lemma \ref{lemma:torsion.DN} $|(\grad v) \circ \gamma|$ on $\S^1$ is analytic in $h$. 
By Lemma \ref{lemma:torsion.DN}, its expansion %Taylor polynomial of order 2 
could be deduced from the formula 
of the shape derivative of the Dirichlet-Neumann operator at zero. 
More precisely, since $G(h) a = 0$ for any constant $a \in \R$, 
one has $G(h) [\frac14 (1+h)^2] = \frac12 G(h)h + \frac14 G(h)(h^2)$, 
%\[
%G(h) \psi 
%= G(h) \Big( \frac{ 1 + 2h + h^2 }{4} \Big) 
%= \frac{G(h)h}{2} + \frac{G(h)(h^2)}{4}
%\]
and the Taylor polynomial of degree 2 of $G(h)h$ is $G(0)h + G'(0)[h] h$, 
while that of $G(h)(h^2)$ is $G(0) (h^2)$. 
Here we propose another way of calculating that polynomial, 
where the shape derivative of the torsion function 
is calculated by differentiating with respect to a parameter $\e$ 
some functions defined on a set $D_\e$ depending on $\e$, 
without explicitly reformulating the problem in terms of functions defined on a fixed 
(i.e., independent of $\e$) domain. 
Note that, by Lemma \ref{lemma:smooth},
the functions and the boundaries involved are smooth. 
All details of the rigorous justification for this approach to shape derivative
%(which is ultimately based on reformulating the problem on a fixed domain)  
can be found, e.g., in \cite{Henrot.Pierre.book}. 
Moreover, since $\ph$ is analytic, its Fréchet and Gateaux derivatives coincide, 
so that, to calculate the Taylor polynomial $T(h,\ph)$, 
it is sufficient to fix $h$ and to calculate the Taylor polynomial 
of the function $\e \mapsto \ph(\e h)$ around $\e = 0$. 
We also observe that the map $h \mapsto (\grad v) \circ \gamma$ 
is analytic because the torsion function satisfies $\grad v = - |\grad v| \nu$ 
on $\pa D$, and both $\ph = |(\grad v) \circ \gamma|$ 
and $\nu \circ \gamma$ are analytic.

\begin{lemma}[Shape derivatives of the torsion function]
\label{lem:derivative}
Let $h\in C^{3,\alpha}(\S^1)$, with $|h| \leq 1/2$. 
For $|\e| < 1$, denote $D_\e$ the nearly circular set of elevation function $\e h$, 
that is, with boundary $\pa D_\e = \{ \gamma_\e(x) : x \in \S^1 \}$, 
where $\gamma_\e(x) := x (1 + \e h(x))$, 
and denote $v_\e$ the torsion function of $D_\e$. 
%and $\e,\e_0$ as in Definition \ref{defn:nearly.spherical}.
Denote 
\[
g(\e) := (\grad v_\e) \circ \gamma_\e 
\quad \text{on } \S^1. 
\]
Then % , on $\S^1$, 
\begin{align}
%\frac{d}{d\e} \Big|_{\e=0} \{ (\nabla v_\e) \circ \g_\e \} 
g'(0) & = \frac{\grad_{\S^1}h}{2} + \frac{ (\mG h) - h }{2} x
\quad \text{on } \S^1,
\label{first.perturbative.derivative.0}
\\
%\la \frac{d^2}{d\e^2} \Big|_{\e=0} \{ (\nabla v_\e) \circ \g_\e \} , \nabla v_0 \ra 
\la g''(0) , g(0) \ra 
& = - \frac{\mG (h^2)}{4} 
+ \frac{ \mG(h \mG h) }{2} 
+ \frac{ h \Delta_{\S^1} h }{2} 
+ \frac{ h \mG h }{2} 
\quad \text{on } \S^1.
\label{second.perturbative.derivative}
\end{align}
\end{lemma}

\begin{proof}
The torsion function of the unit ball $v_0\colon\,D_0\to \R$ is explicitly given by the formula
\begin{equation} \label{explicit.solution.Laplacian.ball}
v_0(x)= \frac{1-|x|^2}{4}
\quad \text{in } D_0.
\end{equation}
Since $v_\e\circ\g_\e(x)=0$ for all $\e \in (-1,1)$, all $x\in\S^1$, 
differentiating this identity in $\e$ we have
\begin{equation}\label{eq:first.derivative}
\dot v_\e \circ \gamma_\e + \la \nabla v_\e \circ \gamma_\e , x \ra h = 0
\quad \text{on } \S^1.
\end{equation}
Differentiating \eqref{eq:first.derivative} again with respect to $\e$, we get
\begin{equation}\label{eq:second.derivative}
\ddot v_\e \circ \gamma_\e  
+ 2 \la (\nabla \dot v_\e) \circ\g_\e , x \ra h 
+ \la ((\nabla^2 v_\e) \circ \g_\e) x, x \ra h^2 = 0
\quad \text{on } \S^1.
\end{equation}
Evaluating \eqref{eq:first.derivative} and \eqref{eq:second.derivative} at $\e=0$, 
on $\S^1$ we obtain, by \eqref{explicit.solution.Laplacian.ball},
\begin{equation}\label{eq:part.derivative}
\dot v_0 = - \la \nabla v_0 , x \ra h = \frac12 h,
\quad \ 
\ddot v_0 = \frac12 h^2 - 2 h \la \nabla \dot v_0, x \ra
\quad \text{on } \S^1.
\end{equation}
We decompose $\nabla \dot v_0$ in its tangential and normal part to find 
\begin{equation}\label{eq:derivative.gradient}
\nabla \dot v_0=\frac{\nabla_{\S^1} h+ (\mG h)x }{2}
\quad \text{on } \S^1,
\end{equation}
where we use the fact that $\Delta \dot v_0 = 0$ in $D_0$ 
and $\dot v_0 = h/2$ on $\pa D_0 = \S^1$, i.e.,   
$\dot v_0$ is the harmonic extension to $D_0$ of the function $\frac12 h$,
so that $\mG (h/2) = \la \grad \dot v_0 , x \ra$ on $\S^1$ 
by the definition of the Dirichlet-Neumann operator $\mG = G(0)$. 
By \eqref{eq:part.derivative} and \eqref{eq:derivative.gradient}, 
we get
\begin{equation} \label{ddot.v0}
\ddot v_0 = \frac{h^2 - 2 h \mG h}{2}, \qquad 
\nabla \ddot v_0 = \frac{ \mG (h^2-2 h \mG h) \, x }{2} + \frac{ \nabla_{\S^1}( h^2-2h \mG h) }{2} 
\quad \text{on } \S^1,
\end{equation}
where we have used the fact that $\Delta \ddot v_0 = 0$ in $D_0$ 
and $\ddot v_0 = (h^2 - 2 h \mG h)/2 =: \psi$ on $\S^1$, i.e., 
$\ddot v_0$ is the harmonic extension of $\psi$ to $D_0$, %  $(h^2 - 2 h \mG h)/2$, 
so that $\la \grad \ddot v_0 , x \ra = \mG \psi$ on $\S^1$, by the definition of $\mG$. 
Differentiating the gradient of $v_\e$ at $\gamma_\e(x)$ with respect to $\e$ gives
\[
g'(\e) = \frac{d}{d\e} \{ (\nabla v_\e) \circ \g_\e \} 
= (\nabla \dot v_\e) \circ \g_\e + h ((\nabla^2 v_\e) \circ \g_\e) x 
\quad \text{on } \S^1.
\]
For $\e = 0$, using \eqref{explicit.solution.Laplacian.ball} and \eqref{eq:derivative.gradient},
we obtain \eqref{first.perturbative.derivative.0}. 
Differentiating again with respect to $\e$,  
\[ %\begin{equation} \label{s}
g''(\e) 
= \frac{d^2}{d\e^2} \{ (\nabla v_\e) \circ \g_\e \} 
= (\nabla \ddot v_\e) \circ \g_\e 
+ 2 h ((\nabla^2 \dot v_\e) \circ \g_\e) x
+ h^2 ((\nabla^3 v_\e) \circ \g_\e) x x,
\] % \end{equation}
and, at $\e = 0$, 
\begin{equation} \label{s.0}
g''(0) % \frac{d^2}{d\e^2} \Big|_{\e = 0} \{ (\nabla v_\e) \circ \g_\e \} 
= \nabla \ddot v_0
+ 2 h (\nabla^2 \dot v_0) x
\quad \text{on } \S^1,
\end{equation}
because $v_0$ is a polynomial of degree 2 in $x$, 
and all its partial derivatives of order 3 vanish. 
The term $\grad \ddot v_0$ on $\S^1$ is in \eqref{ddot.v0}. 
To compute the Hessian of $\dot v_0$, we recall that 
any function $u\in C^2(\overline D_0)$ satisfies 
\begin{equation}\label{eq:laplacian}
\Delta u = \Delta_{\S^1} u + \la \nabla u , x \ra + \la (\nabla^2 u) x , x \ra 
\quad \text{on } \S^1.
\end{equation}
By standard elliptic regularity theory, it is immediate to show that $h\in C^{3.\alpha}(\S^1)$ 
implies that %$\dot v_0 \in C^{2}(\overline D_0)$, because at the boundary it holds 
$\dot v_0\in  C^{2,\alpha}(\overline D_0)$.
Therefore identity \eqref{eq:laplacian} applies to the function $\dot v_0$, and we get
\begin{equation} \label{Hess.dot.v0}
\la (\nabla^2 \dot v_0) x, x \ra 
= - \Delta_{\S^1} \dot v_0 - \la \grad \dot v_0 , x \ra 
= - \frac{ \Delta_{\S^1} h + \mG h }{2} 
\quad \text{on } \S^1,
\end{equation}
because $\Delta \dot v_0 = 0$, 
while $\dot v_0$ and $\grad \dot v_0$ on $\S^1$ are given by 
\eqref{eq:part.derivative}, \eqref{eq:derivative.gradient}. 
By \eqref{s.0}, \eqref{ddot.v0}, \eqref{Hess.dot.v0}, we obtain 
\[
\la g''(0) , x \ra 
= \frac{\mG (h^2)}{2} - \mG(h \mG h) - h \Delta_{\S^1} h - h \mG h 
\quad \text{on } \S^1.
\]
Since $\grad v_0 = -x/2$, the lemma is proved. 
\end{proof}

\begin{lemma}[Taylor polynomial of $\ph$] 
\label{cor:v.t.taylor.exp}
The Taylor polynomial of $\ph^2$ is 
\begin{equation} \label{grad.square.exp}
T(h,\ph^2) = \frac14 + \mF_1(h) + \mF_2(h),
\end{equation} 
that of $\ph^q$, $q \in \R$, is 
\begin{equation} \label{Taylor.ph.q}
T(h, \ph^q) = \frac{1}{2^q} \Big[ 1 + 2 q \mF_1(h) 
+ 2 q \mF_2(h) + 2 q (q-2) \mF_1^2(h) \Big],
\end{equation}
where 
\begin{align}
\mF_1(h) & = \frac{h-\mG h}{2}, 
\label{eq:def.mA.1}
\\
\mF_2(h) & = \frac14 \Big( 
- \mG(h^2) 
+ 2 \mG (h \mG h) 
+ 2 h \Delta_{\S^1} h 
+ |\grad_{\S^1} h|^2
+ (\mG h)^2
+ h^2
\Big).
\label{eq:def.mA.2}
\end{align}
\end{lemma}

\begin{proof} 
One has 
\[
\ph^2(\e h) 
= | (\grad v_\e) \circ \gamma_\e |^2 
= \la (\grad v_\e) \circ \gamma_\e , (\grad v_\e) \circ \gamma_\e \ra 
= \la g(\e) , g(\e) \ra, 
\]
whose Taylor polynomial of degree 2 in $\e$ is 
\[
|g(0)|^2 + 2 \la g(0), g'(0) \ra \e 
+ \big( \la g''(0) , g(0) \ra + |g'(0)|^2 \big) \e^2.
\]
Then use Lemma \ref{lem:derivative}.  
The Taylor polynomial of $\ph^q$ is deduced from the fact that 
the Taylor polynomial of 
$(1 + \e a + \e^2 b)^p$ is $1 + \e p a + \e^2 (pb + \frac12 p (p-1) a^2)$.
\end{proof}

Now, we calculate the Taylor expansion of some integrals. 
Before doing this, we prove the following useful formulae.

\begin{lemma}[Integral identities for sets of normalized area] 
\label{lemma:useful.integrals} 
Suppose that $|D|=|D_0|=\pi$. Then 
\begin{align}
\int_{\S^1}h\,d\s
& = - \frac12 \int_{\S^1} h^2 \,d\s,
\label{h.integral}
\\
\int_{\S^1}\mF_1(h)\,d\s
& = -\frac14 \int_{\S^1} h^2 \,d\s, 
\label{A.1.integral}
\\
\int_{\S^1}\mF_2(h)\,d\s 
& = \frac14 \int_{\S^1} \Big( (\mG h)^2 + h^2 - |\grad_{\S^1}h|^2 \Big) \,d\s. 
\label{A.2.integral}
\end{align}
\end{lemma}

\begin{proof} 
By the assumption $|D|=\pi$, the divergence theorem, 
%\eqref{D.e.def} 
the change of variable $x = \gamma(y)$,  
and \eqref{nu.epsilon.parametrization}, we get
\begin{align*}
\int_{\S^1}1\,d\s
& = 2\pi
= \int_{D} \div x\,dx 
= \int_{\pa D} \la x, \nu_{\pa D} \ra \,d\sigma 
= \int_{\S^1} (1+h)^2\,d\s,
\end{align*}
which gives \eqref{h.integral}. 
Identities \eqref{A.1.integral}, \eqref{A.2.integral} hold by the definition 
\eqref{eq:def.mA.1}, \eqref{eq:def.mA.2} of $\mF_1, \mF_2$  
and because $\mG f$ has zero average, for any function $f$. 
%Similarly, we get \eqref{A.2.integral}.
\end{proof}

\begin{lemma}[Taylor polynomials of $\mA$ and $\mB$] 
\label{lemma:integral.perturbative.exp} 
Assume that $|D| = |D_0| = \pi$. Then the Taylor polynomials of degree 2 of $\mA(D)$ 
and $\mB(D)$ in \eqref{eq:defn.mA} considered as functions of the elevation function $h$ are 
\begin{align}
T(h, \mA)  
& = 
- 4 \pi \int_{\S^1} h \mG h \, d\sigma
+ 2 \pi \int_{\S^1} (\mG h)^2 \, d\sigma
- 2 \pi \int_{\S^1} (\Delta_{\S^1}h) \mG h \, d\sigma, 
\label{A.perturbative.expansion}
\\
T(h, \mB) 
& = \frac{3\pi}{4} \int_{\S^1} h \mG h \, d\sigma 
- \frac{3\pi}{4} \int_{\S^1} (\mG h)^2 \, d\sigma.
\label{B.perturbative.expansion}
\end{align}
\end{lemma}

\begin{proof}  
We calculate the Taylor polynomial of $P, Q, I_R, I_C, I_H$. 
With the change of variable $x = \gamma(y)$, 
the perimeter of $D$ is the integral of $J$ along $\S^1$.
The Taylor polynomial of $J$ is in \eqref{J.expansion}. % $T(h, J) = 1 + h + \frac{1}{2} |\grad_{\S^1} h|^2$, 
Integrating on $\S^1$, and using \eqref{h.integral}, we obtain that the Taylor polynomial of $P$ is 
\begin{equation}\label{perimeter.exp.}
T(h, P) = 2\pi 
- \frac12 \int_{\S^1} h^2 \,d\sigma
+ \frac12 \int_{\S^1} |\grad_{\S^1}h|^2 \, d\sigma.
\end{equation}
Next, we consider the integral $I_R$ in \eqref{def.Q.IR}.
By \eqref{nu.epsilon.parametrization}, one has $\la \gamma , \nu \circ \gamma \ra J = (1+h)^2$, 
and, with the change of variable $x = \gamma(y)$, 
$I_R$ is the integral of $(1+h)^2 \mgt^{-1}$ along $\S^1$, 
where $\mgt = |(\grad v) \circ \gamma|$, see \eqref{def.mgt}. 
The Taylor polynomial of $\ph^{-1}$ is in \eqref{Taylor.ph.q}. 
Integrating on $\S^1$, and using Lemmas \ref{cor:v.t.taylor.exp} and \ref{lemma:useful.integrals},
we obtain 
\begin{align} 
T(h, I_R) 
& = 4\pi 
- \int_{\S^1} h^2  \, d\sigma
+ \int_{\S^1} |\grad_{\S^1}h|^2 \, d\sigma
- 2 \int_{\S^1} h \mG h \, d\sigma
+ 2 \int_{\S^1} (\mG h)^2 \, d\sigma.
\label{Taylor.IR}  % \label{normal.der.integral.exp.}
\end{align}
Next, we consider the integral $I_H$ in \eqref{def.IC.IH}. 
With the change of variable $x = \gamma(y)$, 
$I_H$ is the integral of the product 
$(H_{\pa D} \circ \gamma) \ph J$ along $\S^1$. 
The Taylor polynomials of these factors are in 
\eqref{J.expansion}, 
\eqref{mean.curvature.perturbative.expansion}, 
\eqref{Taylor.ph.q}. 
Integrating on $\S^1$, and using Lemmas \ref{cor:v.t.taylor.exp} and \ref{lemma:useful.integrals},
and the fact that $\Delta_{\S^1}h$ has vanishing integral, we obtain the Taylor polynomial 
\begin{equation}  \label{Taylor.IH}
T(h,I_H) = \pi 
- \frac14 \int_{\S^1} h^2 \, d\sigma
+ \frac14 \int_{\S^1} |\grad_{\S^1} h|^2 \, d\sigma
+ \frac12 \int_{\S^1} h \mG h \, d\sigma 
+ \frac12 \int_{\S^1} (\Delta_{\S^1} h) \mG h \, d\sigma. 
\end{equation}
Now we consider the integral $Q$ in \eqref{def.Q.IR}. 
By the divergence theorem, one proves that $Q$ can be written as the integral 
along the boundary 
\[
Q = \frac14 \int_{\pa D} |\grad v|^2 \la x , \nu \ra \, d\sigma,
\]
see equation (4.30) in \cite{BLL.1}. 
Then we proceed like we did for $I_R$, so that $Q$ is $\frac14$ times the integral 
of $(1+h)^2 \ph^2$ along $\S^1$, and 
\begin{equation}  \label{Taylor.Q}
T(h,Q) = 
\frac{\pi}{8}  
+ \frac14 \int_{\S^1} h^2 \, d\sigma
- \frac{1}{16} \int_{\S^1} |\grad_{\S^1} h|^2 \, d\sigma
- \frac14 \int_{\S^1} h \mG h \, d\sigma 
+ \frac{1}{16} \int_{\S^1} (\mG h)^2 \, d\sigma. 
\end{equation}
Similarly, $I_C$ in \eqref{def.IC.IH} is the integral of $\ph^3 J$ along $\S^1$, and 
\begin{equation} \label{Taylor.IC}
T(h,I_C) 
= \frac{\pi}{4}  
+ \frac12 \int_{\S^1} h^2 \, d\sigma
- \frac{1}{8} \int_{\S^1} |\grad_{\S^1} h|^2 \, d\sigma
- \frac34 \int_{\S^1} h \mG h \, d\sigma 
+ \frac{3}{8} \int_{\S^1} (\mG h)^2 \, d\sigma. 
\end{equation}
The lemma follows from the definition of $\mA$ and $\mB$ in \eqref{eq:defn.mA}.    
\end{proof}

%
%We are now ready to state the main result of this section. 
%\begin{corollary}\label{rem:A.B.signs}  
%Let $D_\e$ be a nearly circular set and
%assume
%\begin{equation}\label{eq:k0geq}
 %k_0>\frac{3}{32}.
%\end{equation}
%Then,
%\begin{equation}\label{eq:statement.final}
%\mathcal C(D) \geq \Big(4k_0-\frac38\Big) |D_\e\triangle B|^2.
%\end{equation}
%\end{corollary}

\begin{proof}[\textbf{Proof of Theorem \ref{thm:rigidity}.}]
We decompose the elevation function $h$ by using spherical harmonics 
(i.e. restriction to the circle $\S^1$ of harmonic homogeneous polinomials 
of two real variables), which provide an orthonormal Hilbert basis of $L^2(\S^1)$. 
Of course parametrizing $\S^1$ with the exponential map
$\T \to \S^1$, $\th \to (\cos \th, \sin \th)$, 
and using classical Fourier series of periodic functions of one real variable 
would be an equivalent and more standard choice;
on the other hand, it would force us to translate all the integrals on $\S^1$ 
into integrals on $\T$, which is an additional step that we can simply avoid. 
 
For any positive integer $\ell$, the real harmonic polynomials of degree $\ell$ in $\R^2$ 
form a vector space of dimension 2. A basis of this space is given by the real and imaginary part 
of the holomorphic function $(x_1 + i x_2)^\ell$.
We denote by $Y_{\ell, m}$, $m=1,2$, their restriction to $\S^1$ 
multiplied by the normalization coefficient $1 / \sqrt{\pi}$
to obtain an orthonormal Hilbert basis of $L^2(\S^1)$.  
For $\ell=0$, there is only $Y_{0,1} = 1 / \sqrt{2\pi}$. 
%We denote $\mI = \{ (0,1) \} \cup \{ (\ell,m) : \ell \geq 1, \ m=1,2 \}$. 
%Using the spherical harmonics, we 
We decompose any function $h \in L^2(\S^1)$ as
\[
h = h_{0,1} Y_{0,1} + \sum_{\begin{subarray}{c} \ell \geq 1 \\ m=1,2 \end{subarray}} % \sum_{(\ell,m) \in \mI} 
h_{\ell, m} Y_{\ell, m}, \quad \ 
h_{\ell,m} = \int_{\S^1} h Y_{\ell,m}\, d\sigma.
\]
For $\ell=0$ and $\ell=1$, we have $Y_{0,1} = 1 / \sqrt{2\pi}$ 
and $Y_{1,m} = x_m / \sqrt{\pi}$, $m=1,2$. 
%\[
%Y_0= \frac{1}{\sqrt{ 2\pi}} \quad \text{and}\quad Y_{1,m}= \frac{x_m}{\sqrt \pi}.
%\]
Therefore
\[
h_{0,1} = \frac{1}{\sqrt {2 \pi}} \int_{\S^1} h\, d\sigma, \quad 
h_{1,m} = \frac{1}{\sqrt \pi} \int_{\S^1} x_m h\, d\sigma, \quad m=1,2.
\]
By assumption, $|D| = |D_0| = \pi$, whence, by \eqref{h.integral}, 
\begin{equation} \label{storz.01}
h_{0,1}^2 = \frac{1}{2\pi} \left(\int_{\S^1} h \, d\sigma \right)^2 
= \frac{1}{2\pi} \left( - \frac12 \int_{\S^1} h^2 \, d\sigma \right)^2. 
\end{equation}
Thus $h_{0,1}^2$ is homogeneous of degree 4 in $h$, 
and it does not contribute to the Taylor polynomials of degree 2. 
We use the barycenter assumption to get rid also of the first order spherical harmonics.
For $m=1,2$, since $\div (x_m x) = 3 x_m$, by the divergence theorem we have  
\begin{equation*}
\int_{D} x_m \, dx 
= \frac13 \int_{D} \div (x_m x) \, dx 
= \frac13 \int_{\pa D} x_m \la x, \nu \ra \, d\sigma. 
\end{equation*}
Since $\la \gamma , \nu \circ \gamma \ra J = (1 + h)^2$, 
making the change of variable $x = \gamma(y)$, 
and renaming $x$ the integration variable, we get 
\[
\frac13 \int_{\pa D} x_m \la x, \nu \ra \, d\sigma
= \frac13 \int_{\S^1} x_m (1+h)^3 \, d\sigma. 
\]
This is $|D|$ times the $m$-th component of the barycenter of the set $D$. 
%is
%\begin{equation*}
%\int_{D_\e}x_i\,dx=\frac13\int_{\pa D_\e}x_i\la x,\nu_\e\ra\,d\mH_1=\frac13\int_{\S^1}x_i(1+\e h)^3 \,d\s.
%\end{equation*}
If the barycenter is at the origin $0$, then 
\begin{equation}\label{eq:ell=1}
\int_{\S^1} x_m h \, d\sigma 
+ \int_{\S^1} x_m h^2 \, d\sigma 
+ \frac13 \int_{\S^1} x_m h^3 \, d\s = 0,
\end{equation}
and $h_{1,m}^2 = \frac{1}{\pi} (\int_{\S^1} x_m h \, d\sigma)^2$ 
does not contribute to the Taylor polynomials of degree 2.  
%Therefore, we have
%\[
%(\e h_{1,m})^2 = o(\e^2),
%\]
Hence in the Taylor polynomials \eqref{A.perturbative.expansion}, \eqref{B.perturbative.expansion},
only coefficients $h_{\ell,m}$ with $\ell \geq 2$ are involved. 

%which gives 
%\[
%\|\e h\|^2_{L^2(\S^1)}= \sum_{\ell\geq 0} \sum_{m} (\e h_{\ell, k})^2= \sum_{\ell\geq 2} \sum_{m=1}^2   (\e h_{\ell, k})^2.
%\]
From the theory of spherical harmonics, one has 
$\mG Y_{\ell, m} = \ell Y_{\ell, m}$, 
$- \Delta_{\S^1} Y_{\ell,m} = \ell^2 Y_{\ell,m}$ 
and $\int_{\S^1} |\nabla_{\S^1} Y_{\ell,m}|^2 \, d\sigma 
= - \int_{\S^1} Y_{\ell,m} \Delta_{\S^1} Y_{\ell,m} \, d\sigma = \ell^2$. 
Thus the Taylor polynomial of degree 2 of the functional $\mC = \mA + \lm_0 \mB$ 
defined in \eqref{def.mC} considered as a function of $h$ is 
\begin{align}
& T(h, \mC) 
\notag \\ 
& = 
- 2 \pi \int_{\S^1} (\Delta_{\S^1} h) \mG h \, d\sigma 
+ \Big(2 \pi - \lm_0 \frac{3 \pi}{4} \Big) \int_{\S^1} (\mG h)^2 \,d\s
+ \Big( - 4 \pi + \lm_0 \frac{3 \pi}{4} \Big) \int_{\S^1} h(\mG h) \,d\s
\notag \\
& = \sum_{\begin{subarray}{c} \ell \geq 2 \\ m=1,2 \end{subarray}}  
\Big[ 2 \pi \ell^3 
+ \Big(2 \pi - \lm_0 \frac{3 \pi}{4} \Big) \ell^2 
+ \Big( - 4 \pi + \lm_0 \frac{3 \pi}{4} \Big) \ell \Big] h_{\ell,m}^2
\notag \\
& = 2 \pi \sum_{\begin{subarray}{c} \ell \geq 2 \\ m=1,2 \end{subarray}} 
\ell (\ell -1) \Big(\ell+2-\frac{3 \lm_0}{8} \Big) h^2_{\ell,m}.
\label{Taylor.mC}
\end{align}
For $\lm_0 < 32/3$, one has $(\ell+2-\frac{3 \lm_0}{8}) \geq (4 - \frac{3 \lm_0}{8}) > 0$ 
for all $\ell \geq 2$. Hence 
\[
\min_{\ell \geq 2} \frac{ 2 \pi \ell (\ell -1) (\ell + 2 - \frac{3 \lm_0}{8}) }{ \ell^3 } 
=: \mu_0 > 0,
\]
and  
\[
T(h,\mC) \geq \mu_0 \| a \|_{H^{\frac32}(\S^1)}^2, 
\]
where $a$ is defined by the decomposition
\begin{equation} \label{def.a.b.h}
a := \sum_{\begin{subarray}{c} \ell \geq 2 \\ m=1,2 \end{subarray}} 
h_{\ell,m} Y_{\ell,m}, 
\quad \ 
b := h_{0,1} Y_{0,1} + \sum_{m=1,2} h_{1,m} Y_{1,m}, 
\quad \ 
h = a+b.
\end{equation}
Since the functions in the integrals defining $\mC$ all satisfy \eqref{est.R.base}, 
the Taylor remainder $R(h,\mC)$ of the functional $\mC$ satisfies 
\[
|R(h,\mC)| \leq C \| h \|_{C^2(\S^1)}^3
\] 
for all $h$ in a ball $\| h \|_{C^2} \leq \delta$. 
To shorten the notation, we omit to indicate that all the norms are on $\S^1$. 
Hence 
\begin{equation}  \label{end.arg.06}
\mC(D) = T(h,\mC) + R(h,\mC) 
\geq \mu_0 \| a \|_{H^{\frac32}}^2 - C \| h \|_{C^2}^3.
%= \mu_0 \| a \|_{H^{\frac32}}^2 
%\Big( 1 - \frac{C \| h \|_{C^2}^3}{\mu_0 \| a \|_{H^{\frac32}}^2} \Big).
\end{equation}
Note that the $C^2$ norm is stronger than the $H^{\frac32}$ norm, 
and their ratio is, in general, unbounded. 
We also note that the bootstrap regularizing effect of Lemma \ref{lemma:smooth} 
concerns regularities at least $C^2$, and its direct application is not sufficient 
to show that that ratio is small. 
To overcome this issue, we proceed in the following way.
By triangular inequality,
\begin{equation} \label{end.arg.04}
\| h \|_{C^2} 
\leq \| a \|_{C^2}
+ \| b \|_{C^2}.
\end{equation}
%First we estimate $a$. 
By Sobolev embedding and interpolation of Sobolev norms 
(see, e.g., \cite{BJL.Dir.Neu}), one has 
\[
\| a \|_{C^2} 
\leq C \| a \|_{H^{3}} 
\leq C \| a \|_{H^{\frac32}}^{\frac23} 
\| a \|_{H^6}^{\frac13},
\]
because $s = (1-\theta) s_0 + \theta s_1$ with 
$s=3$, 
$s_0 = \frac32$, 
$s_1 = 6$,
$\theta = \frac13$.
By the definition of $a$ and Lemma \ref{lemma:smooth}, 
\[
\| a \|_{H^6} 
\leq \| h \|_{H^6} 
\leq C \| h \|_{C^6} 
\leq C \| h \|_{C^2}
\]
(%with a usual abuse of notation, 
different constants are denoted with the same letter $C$).
Hence 
\begin{equation} \label{end.arg.01}
\| a \|_{C^2} 
\leq C \| a \|_{H^{\frac32}}^{\frac23} 
\| h \|_{C^2}^{\frac13}.
\end{equation}
Now we estimate $b$. 
By its definition in \eqref{def.a.b.h}, 
\begin{align}
\| b \|_{C^2} 
& \leq |h_{0,1}| \| Y_{0,1} \|_{C^2} 
+ |h_{1,1}| \| Y_{1,1} \|_{C^2}  
+ |h_{1,2}| \| Y_{1,2} \|_{C^2}  
\notag \\
& \leq C (|h_{0,1}| + |h_{1,1}| + |h_{1,2}|) 
\notag \\ 
& \leq C (|h_{0,1}|^2 + |h_{1,1}|^2 + |h_{1,2}|^2)^{\frac12} 
= C \| b \|_{L^2}. 
\label{end.arg.02}
\end{align}  
By \eqref{storz.01} and \eqref{eq:ell=1}, 
using the pointwise estimates 
\[
|x_m| \leq 1, \quad \ 
|h|^3 \leq |h|^2 \| h \|_{C^0} \leq C |h|^2
\]  
in the integrals, we obtain the quadratic estimates
\[
|h_{0,1}| \leq C \| h \|_{L^2}^2, \quad \ 
|h_{1,m}| \leq C \| h \|_{L^2}^2, \quad m=1,2,
\]
whence 
\[
\| b \|_{L^2} = (|h_{0,1}|^2 + |h_{1,1}|^2 + |h_{1,2}|^2)^{\frac12}
\leq C \| h \|_{L^2}^2.
\]
Since $a$ and $b$ are orthogonal in $L^2(\S^1)$, 
the last inequality becomes % one has  
% OLD: \| h \|_{L^2}^2 = \| a \|_{L^2}^2 + \| b \|_{L^2}^2,
\[
\| b \|_{L^2} 
\leq C \| h \|_{L^2}^2
= C (\| a \|_{L^2}^2 + \| b \|_{L^2}^2).
\]
%Hence 
%\[
%\| b \|_{L^2} (1 - C \| b \|_{L^2}) 
%\leq C \| a \|_{L^2}^2.
%\]
For $\| b \|_{L^2} \leq \| h \|_{L^2} \leq \delta$ sufficiently small, 
the quadratic term $C \| b \|_{L^2}^2$ can be absorbed by the linear term $\| b \|_{L^2}$, 
and we get 
\begin{equation} \label{end.arg.03}
\| b \|_{L^2} \leq C \| a \|_{L^2}^2.
\end{equation}
By \eqref{end.arg.04}, \eqref{end.arg.01} and \eqref{end.arg.03}, 
\begin{equation} \label{end.arg.05}
\| h \|_{C^2} 
\leq \| a \|_{C^2} + C \| a \|_{L^2}^2 
\leq \| a \|_{C^2} (1 + C \| a \|_{L^2})
\leq C \| a \|_{C^2},
\end{equation}
because $\| a \|_{L^2} \leq \| h \|_{L^2} \leq C$. 
Thus, by \eqref{end.arg.05} and \eqref{end.arg.01}, 
\[
\| h \|_{C^2}^3 
\leq C \| a \|_{C^2}^3 
\leq C \| a \|_{H^{\frac32}}^2 \| h \|_{C^2},
\]
and therefore, by \eqref{end.arg.06}, 
\[
\mC(D) 
\geq ( \mu_0 - C \| h \|_{C^2} ) \| a \|_{H^{\frac32}}^2
\geq \frac{\mu_0}{2} \| a \|_{H^{\frac32}}^2
\]
for $\| h \|_{C^2} < \delta$ with $\delta$ sufficiently small. 
On the other hand, by \eqref{main.identity}, $\mC(D) = 0$. This implies that $a = 0$, 
and, by \eqref{end.arg.03}, also $b=0$, namely $h = 0$, and $D$ is a disc. 

Since $D$ is a disc, $\Om$ and $u$ are easily determined, see Remark 4.4 in \cite{BLL.1}.
The proof of Theorem \ref{thm:rigidity} is complete. 
\end{proof}

\bigskip

\begin{flushright}

\textbf{Pietro Baldi}

Dipartimento di Matematica e Applicazioni ``R. Caccioppoli''

University of Naples Federico II

Via Cintia, Monte Sant'Angelo, 80126 Naples, Italy

pietro.baldi@unina.it

\medskip

\textbf{Domenico Angelo La Manna}

Dipartimento di Matematica e Applicazioni ``R. Caccioppoli''

University of Naples Federico II

Via Cintia, Monte Sant'Angelo, 80126 Naples, Italy

domenicoangelo.lamanna@unina.it

\medskip

\textbf{Giuseppe La Scala}

Mathematical and Physical Sciences for Advanced Materials and Technologies

Scuola Superiore Meridionale

Via Mezzocannone, 4, 80138 Naples, Italy

giuseppe.lascala-ssm@unina.it

\end{flushright}

\end{document}